\documentclass[11pt]{article}

\usepackage{amsmath, amssymb, amsthm, calc}  

\title{A simple proof of Schmidt-Summerer's inequality.
       \thanks{ This research was  supported
                by the grant of the President of Russian Federation                
                $\textup N^\circ$ MK--5016.2012.1,
                by RFBR grant $\textup N^\circ$ 12-01-00681a
                and
                by the grant NSH-2519.2012.1 }}
\author{Oleg\,N.\,German, Nikolay\,G.\,Moshchevitin}
\date{}

\theoremstyle{definition}
\newtheorem{definition}{Definition}

\theoremstyle{remark}
\newtheorem{remark}{Remark}

\theoremstyle{plain}
\newtheorem{theorem}{Theorem}
\newtheorem{lemma}{Lemma}

\newtheorem{proposition}{Proposition}
\newtheorem{corollary}{Corollary}

\newtheorem{classic}{Theorem}
\newtheorem{classicprime}[classic]{Theorem}

\DeclareMathOperator{\spanned}{span}

\renewcommand{\phi}{\varphi}

\renewcommand{\vec}[1]{\mathbf{#1}}
\renewcommand{\geq}{\geqslant}
\renewcommand{\leq}{\leqslant}

\newcommand{\e}{\varepsilon}

\newcommand{\R}{\mathbb{R}}
\newcommand{\Z}{\mathbb{Z}}
\newcommand{\Q}{\mathbb{Q}}

\newcommand{\La}{\Lambda}

\newcommand{\bpsi}{\underline{\psi}}
\newcommand{\apsi}{\overline{\psi}}

\newcommand{\cB}{\mathcal{B}}

\newcommand{\cP}{\mathcal{P}}

\newcommand{\gT}{\mathfrak{T}}

\newcommand{\tr}[1]{{#1}^\intercal}

\begin{document}

  \maketitle

  \begin{abstract}
    In this paper we give a simple proof of an inequality for intermediate Diophantine exponents obtained recently by W.\,M.\,Schmidt and L.\,Summerer.
  \end{abstract}

  \section{Introduction} \label{sec:intro}

  Let $\Lambda$ be a unimodular $d$-dimensional lattice in $\R^d$. Denote by $\cB_\infty^d$ the unit ball in sup-norm, i.e. the cube with vertices at the points $(\pm1,\ldots,\pm1)$. Let $G_t:\R^d\to\R^d$ be a map defined by
  \[ G_t(\tr{(z_1,...,z_d)})=\tr{(t^{d-1}z_1, t^{-1}z_2,\ldots,t^{-1}z_d)}. \]
  W.\,M.\,Schmidt and L.\,Summerer \cite{schmidt_summerer,schmidt_summerer_II} studied the asymptotic behaviour of the successive minima of the body $G_t\cB_\infty^d$ with respect to the given lattice $\La$. An appropriate choice of $\La$ connects this setting with the classical setting of simultaneous Diophantine approximation.

  In \cite{schmidt_summerer_II} Schmidt and Summerer proved important inequalities connecting the asymptotics of the first and the $p$-th successive minima, which lead them to an improvement of a famous Jarn\'{\i}k's inequality between the uniform and the ordinary Diophantine exponents \cite{J}. However, the proof they proposed was rather difficult. It uses Mahler's theory of compound bodies and involves a complicated, cumbersome analysis of special piecewise linear functions.

  In the present paper we give a short proof of the main result of \cite{schmidt_summerer_II}. It relies on a simple geometric observation (see Lemma \ref{l:core} below) and does not use the theory of compounds.

  \section{Schmidt-Summerer's exponents} \label{sec:schmex}

  \subsection{General definition} \label{sec:gen_def}

  Let us first give a general definition of the Diophantine exponents we shall deal with throughout the paper. We shall be actually interested in two partial cases, which correspond to the problem of simultaneous Diophantine approximation and to the dual problem, i.e. approximating zero with the values of a linear form.

  For each $d$-tuple $\pmb\tau=(\tau_1,\ldots,\tau_d)\in\R^d$ denote by $D_{\pmb\tau}$ the diagonal $d\times d$ matrix with $e^{\tau_1},\ldots,e^{\tau_d}$ on the main diagonal.
  Let us also denote by $\lambda_p(M)$ the $p$-th successive minimum of a compact symmetric convex body $M\subset\R^d$ (centered at the origin) with respect to the lattice $\La$.

  Let $\gT$ be a path in $\R^d$ defined as $\pmb\tau=\pmb\tau(s)$, $s\in\R_+$, such that
  \begin{equation} \label{eq:sum_of_taus_is_zero}
    \tau_1(s)+\ldots+\tau_d(s)=0,\quad\text{ for all }s.
  \end{equation}

  Set $\cB(s)=D_{\pmb\tau(s)}\cB_\infty^d$. For each $p=1,\ldots,d$\, let us consider the functions
  \[ \psi_p(\La,\gT,s)=\frac{\ln(\lambda_p(\cB(s)))}{s}\,. \]

  \begin{definition} \label{def:schmidt_psi}
    We call the quantity
    \[ \bpsi_p(\La,\gT)=\liminf_{s\to+\infty}\psi_p(\La,\gT,s),\qquad
       \apsi_p(\La,\gT)=\limsup_{s\to+\infty}\psi_p(\La,\gT,s) \]
    \emph{the $p$-th lower} and \emph{upper Schmidt-Summerer's exponents}, respectively.
  \end{definition}

  When it is clear from the context what lattice and what path are under consideration, we shall write simply $\psi_p(s)$, $\bpsi_p$, and $\apsi_p$.

  \subsection{Connection to intermediate Diophantine exponents} \label{sec:inter_exp}

  Given an $n\times m$ real matrix $\Theta$, let us set
  \begin{equation} \label{eq:lattice}
    T_\Theta=
    \begin{pmatrix}
      E_m & 0 \\
      \Theta & E_n
    \end{pmatrix}\qquad\text{ and }\qquad \La_\Theta=T_\Theta^{-1}\Z^d,
  \end{equation}
  where $E_m$ and $E_n$ are the corresponding unity matrices. Let us define $\gT_\Theta:s\mapsto\pmb\tau(s)$ by
  \begin{equation} \label{eq:path}
    \tau_1(s)=\ldots=\tau_m(s)=s,\quad\tau_{m+1}(s)=\ldots=\tau_d(s)=-ms/n.
  \end{equation}

  As it was shown in \cite{german_AA}, Schmidt-Summerer's exponents for $\La=\La_\Theta$ and $\gT=\gT_\Theta$ are closely connected to the intermediate Diophantine exponents of $\Theta$.

  \begin{definition} \label{def:belpha_p}
    The supremum of the real numbers $\gamma$, such that there are arbitrarily large values of $t$ for which (resp. such that for every $t$ large enough) the system of inequalities
    \begin{equation} \label{eq:belpha_p_definition}
      |\vec x|\leq t,\qquad|\Theta\vec x-\vec y|\leq t^{-\gamma}
    \end{equation}
    has $p$ solutions $\vec z_i=(\vec x_i,\vec y_i)\in\Z^m\oplus\Z^n$, $i=1,\ldots,p$, linearly independent over $\Z$, is called the \emph{$p$-th regular (resp. uniform) Diophantine exponent} of $\Theta$ and is denoted by $\beta_p$ (resp. $\alpha_p$).
  \end{definition}

  Namely, in \cite{german_AA}, the following relation was proved for $\bpsi_p=\bpsi_p(\La_\Theta,\gT_\Theta)$, $\apsi_p=\apsi_p(\La_\Theta,\gT_\Theta)$.

  \begin{proposition} \label{prop:belpha_via_psis}
    We have $(1+\beta_p)(1+\bpsi_p)=(1+\alpha_p)(1+\apsi_p)=d/n$.
  \end{proposition}

  It follows immediately from Definition \ref{def:belpha_p} that $\beta_p\geq\alpha_p\geq0$. Combining these inequalities with Proposition \ref{prop:belpha_via_psis}, we get the following trivial lower and upper bounds for $\bpsi_p$ and $\apsi_p$.

  \begin{proposition} \label{prop:bounds}
    We have $-1\leq\bpsi_p\leq\apsi_p\leq m/n$.
  \end{proposition}

  As we said, these bounds are trivial, and we claim that more accurate ones can be obtained. But that is a matter of another research.

  \section{Schmidt-Summerer's inequalities} \label{sec:schmalities}

  The main result of Schmidt and Summerer's paper \cite{schmidt_summerer_II} can be formulated as follows.

  \begin{theorem} \label{t:main}
    Let $\La=\La_\Theta$, $\gT=\gT_\Theta$, where $\La_\Theta$, $\gT_\Theta$ are defined by \eqref{eq:lattice} and \eqref{eq:path}. Then, for $m=1$ and any $p\in\Z$, $1\leq p\leq d$, we have
    \begin{equation} \label{eq:main_1}
      (1+\bpsi_p)(1/n-\apsi_p)\leq(1+\bpsi_1)(1/n-\bpsi_p)
    \end{equation}
    and
    \begin{equation} \label{eq:main_2}
      (1+\apsi_d)(1/n-\apsi_p)\leq(1+\apsi_p)(1/n-\bpsi_p),
    \end{equation}
    provided that $1,\theta_1,\dots,\theta_n$ are linearly independent over $\Q$, where $\theta_1,\dots,\theta_n$ are the components of $\Theta$.
  \end{theorem}

  \begin{remark} \label{rem:dim}
    We actually prove the first half of Theorem \ref{t:main}, i.e. inequality \eqref{eq:main_1}, within a bit weaker assumption than linear independence of $1,\theta_1,\dots,\theta_n$. It is enough to assume (see the beginning of Section \ref{sec:choice_1}) that
    \begin{equation} \label{eq:main_dim}
      \dim_{\Q}\spanned_{\Q}(1,\theta_1,\dots,\theta_n)\geq p.
    \end{equation}
  \end{remark}

  \section{Proof of Theorem \ref{t:main}} \label{sec:proof}

  In this Section we prove Theorem \ref{t:main} way much simpler than it was proved in the original paper \cite{schmidt_summerer_II}. First we make an observation of a local nature. Then we apply this observation for two choices of $\La$ and $\gT$: the one defined by \eqref{eq:lattice}, \eqref{eq:path}, to obtain \eqref{eq:main_1}; and the dual one, to obtain a somewhat dual statement. After that we apply a transference argument basing on a relation proved in \cite{german_AA} to return from the dual statement to $\La_\Theta$ and $\gT_\Theta$, and thus obtain \eqref{eq:main_2}.

  \subsection{Main local observation} \label{sec:local}

  In the following Lemma \ref{l:core} we describe a rather simple geometric phenomenon, after knowing which proving Theorem \ref{t:main} is a matter of technique partially developed in \cite{german_AA}.

  \begin{lemma} \label{l:core}
    Let $\La$ be an arbitrary lattice in $\R^d$ and let $h_1,\ldots,h_d,\lambda$ be arbitrary positive real numbers, $\lambda\geq1$. Consider three parallelepipeds
    \[ \begin{aligned}
         \cP_1 & =\Big\{ \vec z=\tr{(z_1,\ldots,z_d)}\in\R^d \,\Big|\, |z_i|\leq h_i,\ i=1,\ldots,d \Big\}, \\
         \cP_2 & =\Big\{ \vec z=\tr{(z_1,\ldots,z_d)}\in\R^d \,\Big|\, |z_i|\leq\lambda h_i,\ i=1,\ldots,d \Big\}=\lambda\cP_1, \\
         \cP_3 & =\Big\{ \vec z=\tr{(z_1,\ldots,z_d)}\in\cP_2 \,\Big|\, |z_1|\leq h_1 \Big\}.
       \end{aligned} \]
    Suppose that $\cP_1$ contains a lattice point $\vec v$ on its boundary with the first coordinate equal to $h_1$ and that $\cP_2$ contains at least $p$ linearly independent points of $\La$. Then the parallelepiped $2\cP_3$ also contains at least $p$ linearly independent points of $\La$.
  \end{lemma}

  \begin{proof}
    Obviously, there are lattice points $\vec v_1,\ldots,\vec v_{p-1}$ in $\cP_2$ such that $\vec v,\vec v_1,\ldots,\vec v_{p-1}$ are linearly independent. Let $\vec v=\tr{(v_1,\ldots,v_d)}$, $\vec v_i=\tr{(v_{i1},\ldots,v_{id})}$, $i=1,\ldots,p-1$. We may suppose that $v_{i1}\geq0$ for each $i$. Set
    \[ \vec v_i'=\vec v_i-\Big[\frac{v_{i1}}{v_1}\Big]\vec v,\qquad i=1,\ldots,p-1. \]
    Let $\vec v_i'=\tr{(v_{i1}',\ldots,v_{id}')}$. Then $0\leq v_{i1}'<v_1=h_1$ and for each $j=2,\ldots,p-1$ we have
    \[ |v_{ij}'|\leq|v_{ij}|+\Big[\frac{v_{i1}}{v_1}\Big]|v_i|<2\lambda h_i, \]
    since
    \[ 0\leq\Big[\frac{v_{i1}}{v_1}\Big]<\lambda. \]
    Thus, the points $\vec v,\vec v_1',\ldots,\vec v_{p-1}'$ are all contained in $2\cP_3$. Clearly, they are linearly independent.
  \end{proof}

  When applying Lemma \ref{l:core} for fixed $\La$ and $\gT$ we shall take as $\cP_1$ and $\cP_2$ a parallelepiped $\cB(s)$ scaled by the factors $\lambda_1(\cB(s))$ and $\lambda_p(\cB(s))$, respectively. Notice that $\lambda_1(\cB(s))\cB(s)$ contains no nonzero lattice point in its interior, but does contain such a point on its boundary. So, the effect described in Lemma \ref{l:core} works in the case when this point appears to be on the facet orthogonal to the first coordinate axis, i.e. on the facet lying in the hyperplane $z_1=\lambda_1(\cB(s))e^{\tau_1(s)}$. We shall call it \emph{the front facet}.

  \begin{corollary} \label{cor:core}
    Let $m=1$. Suppose $\La$ is an arbitrary lattice and $\gT$ is an arbitrary path such that $\tau_2(s)=\ldots=\tau_d(s)=-\tau_1(s)/n$ for all $s$. Then, for each $s_0$ such that $\lambda_1(\cB(s_0))\cB(s_0)$ contains a lattice point on its front facet, we have
    \begin{equation} \label{eq:core}
      1\leq\frac{e^{\tau_1(s_1)}\lambda_p(\cB(s_1))}{e^{\tau_1(s_0)}\lambda_1(\cB(s_0))}\leq2\qquad\text{ and }\qquad
      1\leq\frac{e^{\tau_i(s_1)}\lambda_p(\cB(s_1))}{e^{\tau_i(s_0)}\lambda_p(\cB(s_0))}\leq2,\qquad i=2,\ldots,d,
    \end{equation}
    where $s_1$ is determined by the relation
    \begin{equation} \label{eq:core_s1_s0}
      e^{\tau_1(s_1)}=e^{\tau_1(s_0)}\left(\frac{\lambda_1(\cB(s_0))}{\lambda_p(\cB(s_0))}\right)^{n/d}.
    \end{equation}
  \end{corollary}

  \begin{proof}
    Set $h_i=\lambda_1(\cB(s_0))e^{\tau_i(s_0)}$, $i=1,\ldots,d$, and $\lambda=\lambda_p(\cB(s_0))/\lambda_1(\cB(s_0))$. Taking into account \eqref{eq:sum_of_taus_is_zero} and \eqref{eq:core_s1_s0} we see that
    \begin{equation} \label{eq:for_the_third}
      e^{\tau_1(s_1)}=\lambda^{n/d}e^{\tau_1(s_0)}\qquad\text{ and }\qquad
      e^{\tau_i(s_1)}=\lambda^{1/d}e^{\tau_i(s_0)},\qquad i=2,\ldots,d.
    \end{equation}
    Define by $h_1,\ldots,h_d,\lambda$ parallelepipeds $\cP_1$, $\cP_2$, $\cP_3$ as in Lemma \ref{l:core}. We have
    \[ \cP_1=\lambda_1(\cB(s_0))\cB(s_0),\quad\cP_2=\lambda_p(\cB(s_0))\cB(s_0),\quad\cP_3=\lambda_1(\cB(s_0))\lambda^{n/d}\cB(s_1). \]
    Indeed, the first two equalities are obvious, the third one follows from \eqref{eq:for_the_third}. The homotheticity of $\cP_3$ and $\cB(s_1)$ is a rather important observation and essentially involves the assumption that $\tau_2(s)=\ldots=\tau_d(s)$ for all $s$.

    By Lemma \ref{l:core} there are at least $p$ linearly independent lattice points in $2\cP_3$. Hence
    \[ \lambda_p(\cB(s_1))\cB(s_1)\subseteq2\cP_3, \]
    which immediately implies the upper bounds in \eqref{eq:core}. The lower ones follow from the fact that $\lambda_p(\cB(s_1))\cB(s_1)$ cannot be a proper subset of $\cP_3$, which is inferred by the inclusion
    \[ \cP_3\subseteq\cP_2=\lambda_p(\cB(s_0))\cB(s_0) \]
    and the fact that the interior of $\lambda_p(\cB(s_0))\cB(s_0)$ does not contain $p$ linearly independent lattice points.
  \end{proof}

  \begin{corollary} \label{cor:core_psied}
    Within the assumptions of Corollary \ref{cor:core}, for each $s_0$ such that $\lambda_1(\cB(s_0))\cB(s_0)$ contains a lattice point on its front facet, we have
    \begin{equation} \label{eq:core_psied}
      \begin{aligned}
        \tau_1(s_0)+s_0\psi_1(s_0)\leq\tau_1(s_1) & +s_1\psi_p(s_1)\leq\tau_1(s_0)+s_0\psi_1(s_0)+\ln2, \\
        \tau_i(s_0)+s_0\psi_p(s_0)\leq\tau_i(s_1) & +s_1\psi_p(s_1)\leq\tau_i(s_0)+s_0\psi_p(s_0)+\ln2,\qquad i=2,\ldots,d,
      \end{aligned}
    \end{equation}
    where $s_1$ is determined by the relation
    \begin{equation} \label{eq:core_s1_s0_psied}
      \tau_1(s_1)=\tau_1(s_0)+\frac{n}{d}s_0\big(\psi_1(s_0)-\psi_p(s_0)\big).
    \end{equation}
  \end{corollary}

  \begin{proof}
    This is an immediate consequence of the definition of $\psi_i(s)$ and Corollary \ref{cor:core}.
  \end{proof}

  \subsection{Auxiliary observation} \label{sec:aux}

  Suppose we are within the assumptions of Corollary \ref{cor:core}. That is $m=1$, $\La$ is an arbitrary lattice and $\gT$ is an arbitrary path such that $\tau_2(s)=\ldots=\tau_d(s)=-\tau_1(s)/n$ for all $s$.

  As we noticed above, all the nonzero lattice points contained in $\lambda_1(\cB(s))\cB(s)$ are gathered on its boundary. Suppose none of them lies on the front facet, i.e. in the hyperplane $z_1=\lambda_1(\cB(s))e^{\tau_1(s)}$. Then we can shrink $\lambda_1(\cB(s))\cB(s)$ along the first coordinate axis until some lattice point lying on the boundary meets the front facet. More precisely, there is a $\mu<1$ such that the parallelepiped
  \[ \cP=\Big\{ \vec z=\tr{(z_1,\ldots,z_d)}\in\R^d \,\Big|\, |z_1|\leq\mu\lambda_1(\cB(s))e^{\tau_1(s)},\ |z_i|\leq\lambda_1(\cB(s))e^{\tau_i(s)},\ i=2,\ldots,d \Big\} \]
  contains no nonzero lattice points in its interior, and does contain such a point on its front facet. Actually this point will lie on the relative boundary of the front facet, i.e. on a face of smaller dimension. It can be easily verified that
  \[ \cP=\mu^{1/d}\lambda_1(\cB(s))\cB(s'), \]
  where $s'$ is determined by the relation $e^{\tau_1(s')}=\mu^{n/d}e^{\tau_1(s)}$. Therefore,
  \[ \lambda_1(\cB(s'))=\mu^{1/d}\lambda_1(\cB(s))<\lambda_1(\cB(s)) \]
  and we arrive at

  \begin{proposition} \label{prop:new_liminf}
    Within the assumptions of Corollary \ref{cor:core} we have
    \[ \bpsi_1(\La,\gT)=\liminf_{s\to+\infty}{}'\,\psi_p(\La,\gT,s), \]
    where $\liminf{}'$ is taken over all $s$ such that $\lambda_1(\cB(s))\cB(s)$ contains a lattice point on its front facet.
  \end{proposition}

  \subsection{The first choice of $\La$ and $\gT$} \label{sec:choice_1}

  Let $m=1$ and $\La=\La_\Theta$, $\gT=\gT_\Theta$. Suppose \eqref{eq:main_dim} is satisfied (cf. Remark \ref{rem:dim}). Let us prove \eqref{eq:main_1}.

  It follows from \eqref{eq:main_dim} that $\lambda_p(\cB(s))e^s\to+\infty$ as $s\to+\infty$. Indeed, if $\lambda_p(\cB(s))e^s$ is bounded, then $\lambda_p(\cB(s))e^{-s/n}$ tends to zero, i.e. the parallelepipeds $\lambda_p(\cB(s))\cB(s)$ degenerate, as $s\to\infty$, into a segment of final length lying in the first coordinate axis, which contradicts the irrationality of this axis with respect to $\La$.

  On the other hand, if along with $\lambda_p(\cB(s))e^s\to+\infty$ we do not have $\lambda_p(\cB(s))e^{-s/n}\to0$ as $s\to+\infty$, then there should exist an $\e>0$, such that the ``tube''
  \[ \Big\{ \vec z=\tr{(z_1,\ldots,z_d)}\in\R^d \,\Big|\, |z_i|\leq\e,\ i=2,\ldots,d \Big\} \]
  contains no $p$ linearly independent lattice points. This is possible only if the first coordinate axis is contained in a subspace of $\R^d$ of dimension less than $p$, which is rational with respect to $\La$. Which contradicts \eqref{eq:main_dim}.

  Thus, we have
  \[ \lambda_p(\cB(s))e^s\to+\infty\quad\text{ and }\quad\lambda_p(\cB(s))e^{-s/n}\to0\quad\text{ as }\quad s\to+\infty, \]
  i.e.
  \begin{equation} \label{eq:tending_to_infty}
    s(1+\psi_p(s))\to+\infty\quad\text{ and }\quad s(1/n-\psi_p(s))\to+\infty\quad\text{ as }s\to+\infty.
  \end{equation}
  Since $\lambda_1(\cB(s))\leq\lambda_p(\cB(s))$, we also have $\lambda_1(\cB(s))e^{-s/n}\to0$ as $s\to+\infty$. Hence there are arbitrarily large values of $s$, such that the parallelepiped $\lambda_1(\cB(s))\cB(s)$ contains a lattice point on its front facet. For each $s_0$ satisfying this condition Corollary \ref{cor:core_psied} gives us $s_1$ such that \eqref{eq:core_psied} and \eqref{eq:core_s1_s0_psied} hold. Applying \eqref{eq:path} we rewrite \eqref{eq:core_psied} and \eqref{eq:core_s1_s0_psied} as
  \begin{equation} \label{eq:core_psied_choice_1}
    \begin{array}{c}
      s_0(1+\psi_1(s_0))\leq s_1(1+\psi_p(s_1))\leq s_0(1+\psi_1(s_0))+\ln2, \\ \vphantom{\frac{\big|}{}}
      s_0(1/n-\psi_p(s_0))-\ln2\leq s_1(1/n-\psi_p(s_1))\leq s_0(1/n-\psi_p(s_0)).
    \end{array}
  \end{equation}
  and
  \begin{equation} \label{eq:core_s1_s0_psied_choice_1}
    s_1=s_0\left(1+\frac{n}{d}\big(\psi_1(s_0)-\psi_p(s_0)\big)\right).
  \end{equation}

  By Proposition \ref{prop:bounds} we have $\bpsi_1-\apsi_p\geq-d/n$ with equality only in case $\bpsi_1=-1$, $\apsi_p=1/n$, when \eqref{eq:main_1} is trivial. So, we may suppose that $\bpsi_1-\apsi_p>-d/n$, i.e. $\psi_1(s_0)-\psi_p(s_0)$ is bounded away from $-d/n$ for large $s_0$. Therefore, \eqref{eq:core_s1_s0_psied_choice_1} implies that
  \begin{equation} \label{eq:s_1_grows}
    s_1\to\infty\quad\text{ as }\quad s_0\to\infty.
  \end{equation}
  It follows from \eqref{eq:tending_to_infty}, \eqref{eq:core_psied_choice_1} and \eqref{eq:s_1_grows} that all the sides of \eqref{eq:core_psied_choice_1} are positive for large $s_0$. Thus, we may conclude from \eqref{eq:core_psied_choice_1} that for large $s_0$
  \begin{equation*} 
    \big(1+\psi_p(s_1)\big)\big(1/n-\psi_p(s_0)-\tfrac{\ln2}{s_0}\big)\leq\big(1/n-\psi_p(s_1)\big)\big(1+\psi_1(s_0)+\tfrac{\ln2}{s_0}\big).
  \end{equation*}
  Applying again \eqref{eq:tending_to_infty}, we get
  \begin{equation} \label{eq:almost_main_1}
    \big(1+\psi_p(s_1)\big)\big(1/n-\psi_p(s_0)\big)\leq\big(1+\psi_1(s_0)\big)\big(1/n-\psi_p(s_1)\big)(1+o(1))\ \text{ as }\ s_0\to\infty.
  \end{equation}
  In view of Proposition \ref{prop:new_liminf} we can choose $s_0$ large enough to guarantee $\psi_1(s_0)$ to be however close to $\bpsi_1$. Thus, in view of the rough estimates
  \[ \psi_p(s_1)\geq\bpsi_p+o(1),\qquad\psi_p(s_0)\leq\apsi_p+o(1),\qquad\text{as }s_0\to\infty, \]
  \eqref{eq:almost_main_1} leads us to \eqref{eq:main_1}.

  \subsection{The second choice of $\La$ and $\gT$} \label{sec:choice_2}

  Let $m=1$ and $\La=\La_\Theta^\ast$, $\gT=\gT_\Theta^\ast$, where
  \begin{equation} \label{eq:dual_lattice}
    \La_\Theta^\ast=T_\tr\Theta\Z^d=
    \begin{pmatrix}
      E_m & \tr\Theta \\
      0 & E_n
    \end{pmatrix}\Z^d
  \end{equation}
  and $\gT_\Theta^\ast:s\mapsto\pmb\tau^\ast(s)$ is defined by
  \begin{equation} \label{eq:path_choice_2}
    \tau_1^\ast(s)=-ns,\quad\tau_2^\ast(s)=\ldots=\tau_d^\ast(s)=s.
  \end{equation}
  Clearly, $\La_\Theta^\ast$ is dual for $\La_\Theta$. Therefore, it follows from linear independence of $1,\theta_1,\ldots,\theta_n$ over $\Q$ that there are no nonzero points of $\La$ with first coordinate equal to zero. Hence
  \[ \lambda_p(\cB(s))e^{-ns}\to0\quad\text{ and }\quad\lambda_p(\cB(s))e^s\to+\infty\quad\text{ as }\quad s\to+\infty, \]
  i.e.
  \begin{equation} \label{eq:tending_to_infty_choice_2}
    s(n-\psi_p(s))\to+\infty\quad\text{ and }\quad s(1+\psi_p(s))\to+\infty\quad\text{ as }s\to+\infty.
  \end{equation}
  Besides that, there are arbitrarily large values of $s$, such that the parallelepiped $\lambda_1(\cB(s))\cB(s)$ contains a lattice point on its front facet. For each $s_0$ satisfying this condition Corollary \ref{cor:core_psied} gives us $s_1$ such that \eqref{eq:core_psied} and \eqref{eq:core_s1_s0_psied} hold. Applying \eqref{eq:path_choice_2} we rewrite \eqref{eq:core_psied} and \eqref{eq:core_s1_s0_psied} as
  \begin{equation} \label{eq:core_psied_choice_2}
    \begin{array}{c}
      s_0(n-\psi_1(s_0))-\ln2\leq s_1(n-\psi_p(s_1))\leq s_0(n-\psi_1(s_0)), \\ \vphantom{\frac{\big|}{}}
      s_0(1+\psi_p(s_0))\leq s_1(1+\psi_p(s_1))\leq s_0(1+\psi_p(s_0))+\ln2.
    \end{array}
  \end{equation}
  and
  \begin{equation} \label{eq:core_s1_s0_psied_choice_2}
    s_1=s_0\Big(1+\frac{1}{d}\big(\psi_p(s_0)-\psi_1(s_0)\big)\Big).
  \end{equation}
  In this case we immediately have
  \begin{equation} \label{eq:s_1_grows_choice_2}
    s_1\to\infty\quad\text{ as }\quad s_0\to\infty,
  \end{equation}
  since $\psi_p(s_0)\geq\psi_1(s_0)$.
  It follows from \eqref{eq:tending_to_infty_choice_2}, \eqref{eq:core_psied_choice_2} and \eqref{eq:s_1_grows_choice_2} that all the sides of \eqref{eq:core_psied_choice_2} are positive for large $s_0$. Thus, we may conclude from \eqref{eq:core_psied_choice_2} that for large $s_0$
  \begin{equation*} 
    \big(1+\psi_p(s_1)\big)\big(n-\psi_1(s_0)-\tfrac{\ln2}{s_0}\big)\leq\big(n-\psi_p(s_1)\big)\big(1+\psi_p(s_0)+\tfrac{\ln2}{s_0}\big).
  \end{equation*}
  Applying again \eqref{eq:tending_to_infty_choice_2}, we get
  \begin{equation} \label{eq:almost_dual_main_2}
    \big(1+\psi_p(s_1)\big)\big(n-\psi_1(s_0)\big)\leq\big(1+\psi_p(s_0)\big)\big(n-\psi_p(s_1)\big)(1+o(1))\ \text{ as }\ s_0\to\infty.
  \end{equation}
  In view of Proposition \ref{prop:new_liminf} we can choose $s_0$ large enough to guarantee $\psi_1(s_0)$ to be however close to $\bpsi_1$. Thus, in view of the rough estimates
  \[ \psi_p(s_1)\geq\bpsi_p+o(1),\qquad\psi_p(s_0)\leq\apsi_p+o(1),\qquad\text{as }s_0\to\infty, \]
  \eqref{eq:almost_dual_main_2} leads us to
  \begin{equation*}
    \big(1+\bpsi_p\big)\big(n-\bpsi_1\big)\leq\big(1+\apsi_p\big)\big(n-\bpsi_p\big),
  \end{equation*}
  or, after excluding references to the context,
  \begin{equation} \label{eq:dual_main_2}
    \big(1+\bpsi_p(\La_\Theta^\ast,\gT_\Theta^\ast)\big)\big(n-\bpsi_1(\La_\Theta^\ast,\gT_\Theta^\ast)\big)\leq
    \big(1+\apsi_p(\La_\Theta^\ast,\gT_\Theta^\ast)\big)\big(n-\bpsi_p(\La_\Theta^\ast,\gT_\Theta^\ast)\big),
  \end{equation}

  \subsection{Transference argument} \label{sec:transference}

  Let $m=1$. In \cite{german_AA} the following is proved.

  \begin{proposition} \label{prop:starred_psis_via_psis_m_is_1}
    We have
    \[ \bpsi_p(\La_\Theta^\ast,\gT_\Theta^\ast)=-n\apsi_{d+1-p}(\La_\Theta,\gT_\Theta)\quad\text{ and }\quad
       \apsi_p(\La_\Theta^\ast,\gT_\Theta^\ast)=-n\bpsi_{d+1-p}(\La_\Theta,\gT_\Theta)\,. \]
  \end{proposition}

  Applying Proposition \ref{prop:starred_psis_via_psis_m_is_1} to \eqref{eq:dual_main_2} we get for each $p=1,\ldots,d$
  \begin{equation} \label{eq:almost_main_2}
    \big(1/n-\apsi_{d+1-p}(\La_\Theta,\gT_\Theta)\big)\big(1+\apsi_d(\La_\Theta,\gT_\Theta)\big)\leq
    \big(1/n-\bpsi_{d+1-p}(\La_\Theta,\gT_\Theta)\big)\big(1+\apsi_{d+1-p}(\La_\Theta,\gT_\Theta)\big).
  \end{equation}
  Since \eqref{eq:almost_main_2} holds for all $p$, we may substitute $d+1-p$ by $p$ and thus obtain \eqref{eq:main_2}.

\vskip 10mm

\noindent
Oleg N. {\sc German} \\
Moscow Lomonosov State University \\
Vorobiovy Gory, GSP--1 \\
119991 Moscow, RUSSIA \\
\emph{E-mail}: {\fontfamily{cmtt}\selectfont german@mech.math.msu.su, german.oleg@gmail.com}

\vskip 10mm

\noindent
Nikolay G. {\sc Moshchevitin} \\
Moscow Lomonosov State University \\
Vorobiovy Gory, GSP--1 \\
119991 Moscow, RUSSIA \\
\emph{E-mail}: {\fontfamily{cmtt}\selectfont moshchevitin@rambler.ru, moshchevitin@gmail.com}

\end{document}